\documentclass{amsart}[12pt]
\usepackage{amsthm}
\usepackage{amssymb}
\usepackage{upref}
\usepackage{amscd}

\usepackage{xy}
\input xy
\xyoption{all}
\newtheorem{theorem}{Theorem}[section]

\newtheorem{proposition}[theorem]{Proposition}

\newtheorem{lemma}[theorem]{Lemma}
\newtheorem{corollary}[theorem]{Corollary}

\newtheorem{main}{Theorem}[section]

\theoremstyle{definition}
\newtheorem{definition}[theorem]{Definition}
\newtheorem{defn}[theorem]{Definition}

\newtheorem{example}[theorem]{Example}

\theoremstyle{remark}
\newtheorem{remark}[theorem]{Remark}

\numberwithin{equation}{section}

\newcommand{\F}{\mathbb{F}}

\newcommand{\M}{\mathcal{M}}
\newcommand{\cN}{\mathcal{N}}

\newcommand{\algo}{Alg_\cat{O}}

\renewcommand{\tilde}[1]{\widetilde{#1}}

\DeclareMathOperator{\Sym}{Sym}

\DeclareMathOperator{\Ho}{Ho}

\DeclareMathOperator{\ho}{Ho}

\newcommand{\po}{\ar@{}[dr]|(.7){\Searrow}}
\newcommand{\pb}{\ar@{}[dr]|(.3){\Nwarrow}}

\DeclareMathOperator{\map}{map}

\newcommand{\cat}[1]{\mathcal{#1}}

\newcommand{\algt}{\mbox{Alg}_T}

\newcommand{\algtlcm}{\algt(L_\cat{C}(\M))}
\newcommand{\algtm}{\algt(\M)}
\newcommand{\ltcalgtm}{L_{T(\cat{C})}\algtm}
\newcommand{\lcm}{L_\cat{C}(\M)}

\newcommand{\lc}{L_{\cat{C}}}
\newcommand{\ltc}{L_{T(\cat{C})}}
\newcommand{\hoalgtm}{\mbox{Ho(Alg}_T(\M))}

\begin{document}

\title{Left Bousfield Localization and Eilenberg-Moore Categories}

\author{Michael Batanin}
\address{Department of Mathematics \\ Macquarie University \\ North Ryde, 2109 Sydney, Australia}
\email{michael.batanin@mq.edu.au}
\thanks{The first author gratefully acknowledges the financial support of Scott Russel Johnson Memorial Foundation  and Australian Research Council (grants  No. DP1095346). }

\author{David White}
\address{Department of Mathematics and Computer Science \\ Denison University
\\ Granville, OH 43023}
\email{david.white@denison.edu}
\thanks{The second author was supported by the National Science Foundation under Grant No. IIA-1414942}

\begin{abstract}
We prove the equivalence of several hypotheses that have appeared recently in the literature for studying left Bousfield localization and algebras over a monad. We find conditions so that there is a model structure for local algebras, so that localization preserves algebras, and so that localization lifts to the level of algebras. We include examples coming from the theory of colored operads, and applications to spaces, spectra, and chain complexes.
\end{abstract}

\maketitle

\tableofcontents

\section{Introduction}

Left Bousfield localization has become a fundamental tool in modern abstract homotopy theory. The ability to take a well-behaved model category and a prescribed set of maps, and then produce a new model structure where those maps are weak equivalences, has applications in a variety of settings. Left Bousfield localization is required to construct modern stable model structures for spectra, including equivariant and motivic spectra. Left Bousfield localization also provides a powerful computational device for studying spaces, spectra, homology theories, and numerous algebraic examples of interest (e.g. in the model categories $Ch(R)$ and $R$-mod that arise when studying homological algebra and the stable module category, respectively).

In recent years, several groups of researchers have been applying the machinery of left Bousfield localization to better understand algebras over (colored) operads, especially results regarding when algebraic structure is preserved by localization, e.g. \cite{batanin-baez-dolan-via-semi}, \cite{batanin-berger}, \cite{blumberg-hill-2}, \cite{cgmv}, \cite{carato}, \cite{gutierrez-rondigs-spitzweck-ostvaer-colocalizations}, \cite{white-localization}, \cite{white-yau}. Unsurprisingly, several different approaches have emerged. In this paper, we will prove that these approaches are equivalent, and will use this equivalence to provide new structural features that may be used in any of these settings.

Our setting will be a monoidal model category $\M$, with an action of a monad $T$, and with a prescribed set of maps $\cat{C}$ that we wish to invert. In our applications, $\M$ will be sufficiently nice that the left Bousfield localization $L_\cat{C}(\M)$ exists, but for much of the paper we do not require this existence. Our main theorem follows (and will be proven as Theorem \ref{thm:omnibus}):

\begin{main} \label{thm:main}
Assume $T$ is a monad on $\M$ such that $\algtm$ has a transferred semi-model structure, and that $\lcm$ is compactly generated. The following are equivalent:
\begin{enumerate}
\item $L_\cat{C}$ lifts to a left Bousfield localization $\ltc$ of semi-model categories on $\algtm$.
\item $U:\algtm \to \M$ preserves local equivalences.
\item $\algtlcm$ has a transferred semi-model structure.
\item $L_\cat{C}$ preserves $T$-algebras.
\end{enumerate}
Furthermore, any of the above implies
\begin{enumerate}
\item[(5)] $T$ preserves $\cat{C}$-local equivalences between cofibrant objects.
\end{enumerate}
\end{main}

This theorem unifies all known approaches to studying the homotopy theory of localization for algebras over a monad. Theorem \ref{thm:main} can be viewed as providing conditions to check to determine if localization preserves algebras, if localization lifts to the level of algebras, and if transferred semi-model structures exist.

In Section \ref{sec:alg-in-loc}, we discuss how to pass from the localized model structure $L_\cat{C}(\M)$ to the category $\algtlcm$, needed for item (3) in the theorem. In Section \ref{sec:loc-in-alg-carato}, we discuss lifting the localization $L_\cat{C}$ to a localization $L_{T(\cat{C})}$ of algebras, where $T(\cat{C})$ denotes the free $T$-algebra maps on $\cat{C}$. This is needed for items (1) and (2) of the theorem. In Section \ref{sec:preservation}, we compare these two approaches to studying local algebras, and we also compare them to item (4) above regarding preservation of $T$-algebras by localization. After proving our main theorem in Section \ref{sec:preservation}, we prove a partial converse, i.e. we determine what can be said if we know item (5) is true. In particular, we show that item (5), together with mild hypotheses on the monad $T$, implies that $U$ reflects local weak equivalences. Numerous examples and applications are given throughout the paper.

\section{Preliminaries}

We assume the reader is familiar with the basics of model categories. For our entire paper, $\M$ will denote a cofibrantly generated (semi-)model category, and $I$ (resp. $J$) will denote the generating cofibrations (resp. generating trivial cofibrations). Semi-model categories are recalled below, and familiarity with semi-model categories is not assumed.

\subsection{Left Bousfield Localization}

Given a model category $\M$ and a set of morphisms $\cat{C}$, the left Bousfield localization of $\M$ with respect to $\cat{C}$ is, if it exists, a new model structure on $\M$, denoted $L_\cat{C}(\M)$, satisfying the following universal property: the identity functor $\M \to L_\cat{C}(\M)$ is a left Quillen functor and for any model category $\cN$ and any left Quillen functor $F:\M \to \cN$ taking the maps $\cat{C}$ to weak equivalences, there is a unique left Quillen functor $L_\cat{C}(\M) \to \cN$ through which $F$ factors.

\begin{defn}
An object $N$ is said to be \textit{$\cat{C}$-local} if it is fibrant in $\M$ and if for all $g:X\to Y$ in $\cat{C}$, the induced map on simplicial mapping spaces $\map(g,N) \!:\map(Y,N)\to \map(X,N)$ is a weak equivalence in $sSet$. These objects are precisely the fibrant objects in $L_\cat{C}(\M)$. A map $f:A\to B$ is a \textit{$\cat{C}$-local equivalence} if for all $N$ as above, $\map(f,N):\map(B,N)\to \map(A,N)$ is a weak equivalence. These maps are precisely the weak equivalences in $L_\cat{C}(\M)$.
\end{defn}

The existence of the model category $L_\cat{C}(\M)$ can be guaranteed by assuming $\M$ is left proper and either cellular or combinatorial \cite{hirschhorn}. We never need these conditions, so we always just assume $L_\cat{C}(\M)$ exists. Throughout this paper, we assume $\cat{C}$ is a set of cofibrations between cofibrant objects. This can always be arranged by taking cofibrant replacements of the maps in $\cat{C}$.

\subsection{Algebras over Monads}

As we will be working with filtrations in order to transfer model structures to categories of algebras, an additional smallness condition will be required (one which would not be necessary if we were to assume $\M$ is combinatorial). The following definition is from \cite{batanin-berger}:

\begin{defn}
Let $\cat{K}$ be a saturated class of morphisms. $\M$ is called \textit{$\cat{K}$-compactly generated}, if all objects are small relative to $\cat{K}$-cell and if the weak equivalences are closed under filtered colimits along morphisms in $\cat{K}$.
\end{defn}

In practice, $\M$ is often a monoidal model category (i.e. satisfies the pushout product axiom as in \cite{SS00}), and $\cat{K}$ is often taken to be the class $(\M \otimes I)$-cell, i.e. the monoidal saturation of the cofibrations. 

We will be interested in transferring model structures to categories of algebras over a monad $T$, where $U$ is the forgetful functor from $T$-algebras to $\M$. We shall always assume $T$ is finitary, i.e. $T$ preserves filtered colimits. 

\begin{defn} \label{defn:transferred}
A model structure will be called \textit{transferred} from $\M$ if a map $f$ of $T$-algebras is a weak equivalence (resp. fibration) if and only if $U(f)$ is a weak equivalence (resp. fibration). We also call this structure on $\algtm$ \textit{projective}.
\end{defn}

Most of the model category axioms for the transferred model structure are easy to check, but the difficult one has to do with proving that the trivial cofibrations of $T$-algebras are saturated and contained in the weak equivalences. These trivial cofibrations are generated by the set $T(J)$, and so pushouts in $\algtm$ of the following sort must be considered:
\begin{align} \label{diagram:pushout}
\xymatrix{T(K)\ar[r]^{T(u)} \ar[d] & T(L)\ar[d]\\
A \ar[r]_h & B}
\end{align}

Often, such pushouts are computed by filtering the map $h:A\to B$ as a transfinite composition of pushouts in $\M$. The following definition ensures this works.

\begin{defn}
Let $\cat{K}$ be a saturated class of morphisms. A monad will be called \textit{$\cat{K}$-admissible} if, for each (trivial) cofibration $u:K\to L$ in $\M$, the pushout (\ref{diagram:pushout}) has $U(h)$ in $\cat{K}$ (resp. in $\cat{K}$ and a weak equivalence). A monad will be called \textit{$\cat{K}$-semi admissible} if this holds for pushouts into cofibrant $T$-algebras $A$. A monad will be called \textit{$\cat{K}$-semi admissible over $\M$} if this holds for pushouts into $T$-algebras $A$ for which $UA$ is cofibrant in $\M$.
\end{defn}

The value of this definition is that it allows for transferred model structures and transferred semi-model structures (to be discussed below) on $\algtm$:

\begin{theorem} \label{thm:admissible-implies-(semi)-models}
Let $T$ be a finitary, $\cat{K}$-admissible (resp. $\cat{K}$-semi-admissible, resp. $\cat{K}$-semi-admissible over $\M$) monad on a $\cat{K}$-compactly generated model category $\M$ for some saturated class of morphisms $\cat{K}$. Then $\algtm$ admits a transferred model structure (resp. semi-model structure, resp. semi-model structure over $\M$).
\end{theorem}

This theorem is proven as Theorem 2.11 of \cite{batanin-berger} for model structures, and as 12.1.4 and 12.1.9 in \cite{fresse-book} for semi-model structures. What we call semi-model structures in this paper are called $J$-semi model structures in \cite{spitzweck-thesis}. What we call semi-model structures over $\M$ (following Spitzweck's terminology) are called relative semi-model structures in \cite{fresse-book}, but note that \cite{fresse-book} has a weaker notation of semi-model structure, called an $(I,J)$-semi model structure in \cite{spitzweck-thesis}. If $\M$ is a monoidal model category, algebras over any $\Sigma$-cofibrant colored operad inherit a transferred semi-model structure \cite{gutierrez-rondigs-spitzweck-ostvaer} [Theorem A.8]. Under mild conditions, algebras over entrywise cofibrant colored operads inherit transferred semi-model structures \cite{white-yau} [Theorem 6.2.3]. In the absence of the monoid axiom, monoids inherit a transferred semi-model structure \cite{hovey-monoidal} [Theorem 3.3]. In the presence of the commutative monoid axiom but the absence of the monoid axiom, commutative monoids inherit a transferred semi-model structure \cite{white-commutative-monoids} [Corollary 3.8]. So semi-admissibility for a monad $T$ occurs much more frequently than admissibility, and we will provide an example at the end of \ref{subsec:semi} of category of algebras that only admits a semi-model structure, not a full model structure.

\subsection{Semi-Model Categories} \label{subsec:semi}

The following definition is from \cite{spitzweck-thesis}, where it is called a $J$-semi model category over $\M$.

\begin{definition}
\label{defn:semi-model-cat}
Assume there is an adjunction $F:\M \rightleftarrows \cat{D}:U$ where $\M$ is a cofibrantly generated model category, $\cat{D}$ is bicomplete, and $U$ preserves colimits over non-empty ordinals. 

We say that $\cat{D}$ is a \textbf{semi-model category} if $\cat{D}$ has three classes of morphisms called \emph{weak equivalences}, \emph{fibrations}, and \emph{cofibrations} such that the following axioms are satisfied. A \emph{cofibrant} object $X$ means an object in $\cat{D}$ such that the map from the initial object of $\cat{D}$ to $X$ is a cofibration in $\cat{D}$.

\begin{enumerate}
\item $U$ preserves and reflects fibrations and trivial fibrations.
\item $\cat{D}$ satisfies the 2-out-of-3 axiom and the retract axiom.
\item Cofibrations in $\cat{D}$ have the left lifting property with respect to trivial fibrations. Trivial cofibrations in $\cat{D}$ whose domain is cofibrant have the left lifting property with respect to fibrations.
\item Every map in $\cat{D}$ can be functorially factored into a cofibration followed by a trivial fibration. Every map in $\cat{D}$ whose domain is cofibrant can be functorially factored into a trivial cofibration followed by a fibration.
\item The initial object in $\cat{D}$ is cofibrant.
\item Fibrations and trivial fibrations are closed under pullback.
\end{enumerate}

Denote by $I'$-inj the class of maps that have the right lifting property with respect to maps in $I'$. $\cat{D}$ is said to be \textit{cofibrantly generated} if there are sets of morphisms $I'$ and $J'$ in $\cat{D}$ such that the following conditions are satisfied. 
\begin{enumerate}
\item
$I'$-inj is the class of trivial fibrations.
\item
$J'$-inj is the class of fibrations in $\cat{D}$.
\item
The domains of $I'$ are small relative to $I'$-cell.
\item
The domains of $J'$ are small relative to maps in $J'$-cell whose domain is cofibrant.
\end{enumerate}

If, in the definition above, cofibrancy in $\cat{D}$ is weakened to cofibrancy in $\M$, then the structure on $\cat{D}$ is called a \textbf{semi-model category over $\M$}. 
\end{definition}

In this paper, we will often transfer a semi-model structure from $\lcm$ to $\algtlcm$. Following Definition \ref{defn:transferred}, this means $U$ preserves and reflects weak equivalences and fibrations. The most general setting for such a transfer to exist is that given by Theorem \ref{thm:admissible-implies-(semi)-models} applied to $\lcm$. We work in this general setting, but the reader is encouraged to keep the following examples in mind.

\begin{example}
Suppose $\M$ is a monoidal model category, e.g. simplicial sets, compactly generated spaces, chain complexes, symmetric spectra, or (equivariant) orthogonal spectra. If $L_\cat{C}(\M)$ satisfies the pushout product axiom, $L_\cat{C}$ is called a \textit{monoidal Bousfield localization}. Conditions to guarantee this are given in \cite{white-thesis} and worked out for the model categories just listed, as well as counterexamples demonstrating it does not come for free.

For monoidal Bousfield localizations, $\algtlcm$ inherits a transferred semi-model structure whenever $T$ comes from a $\Sigma$-cofibrant colored operad. Similar results hold for entrywise cofibrant colored operads \cite{white-yau}. Results of this nature have been proven for commutative monoids in \cite{white-commutative-monoids}, and are recalled in Example \ref{example:comm-mon-sym}. In \cite{white-yau-coloc}, conditions are given so that such transfers exist in a right Bousfield localization $R_K(\M)$.
\end{example}

Note that another source of semi-model categories is as left Bousfield localizations of other semi-model categories. This theory has been worked out by the second author in \cite{white-bous-loc-semi}. We mention the results of \cite{white-bous-loc-semi} at various places in this document, but never require them for any of our proofs.

We conclude this section with an example of a semi-model structure that is not a full model structure, to demonstrate that semi-model categories are inescapable when one studies algebras over operads.

\begin{example}
Consider the colored operad $P$ whose algebras are non-reduced symmetric operads. Consider $P$-algebras in $Ch(\F_2)$ with the projective model structure, from \cite{hovey-book}. Because $P$ is $\Sigma$-cofibrant and $\M$ is a cofibrantly generated monoidal model category, there is a transferred semi-model structure on $P$-alg, by Theorem 6.3.1 in \cite{white-yau} (this result goes back to \cite{spitzweck-thesis}). This transferred structure does not form a full model structure, as we now show. Let Com be the operad for commutative differential graded algebras, so Com(n) is $\F_2$ with the trivial $\Sigma_n$-action for all $n$.

For any acyclic complex $C$, the inclusion map $0\to C$ is a trivial cofibration. Define a collection $K$ to have $K_0 = C$ and $K_i = 0$ otherwise. Then the inclusion from the 0 collection to $K$ is a trivial cofibration. Thus, $P(0)\to P(K)$ is supposed to be a trivial cofibration, so pushouts of it would need to be weak equivalences if $P$-alg was a model category. Yet in
\begin{align*}
\xymatrix{P(0)\ar[r] \ar[d] & P(K) \ar[d] \\ Com \ar[r] & P(K)\coprod Com}
\end{align*}
there are choices of $C$ such that the bottom map is not a weak equivalence, because it contains a summand of $C\otimes C / \Sigma_2$. We now demonstrate an explicit example of $C$ such that $H_1(C\otimes C / \Sigma_2) \cong \F_2$ is not contractible. To ease notation, let $k = \F_2$.

Let C be the complex $0 \to k \to k \oplus k \to k\oplus k \to k \to 0$ where the differential $d_3$ takes a to (a,a), $d_2$ takes (a,b) to (a+b,a+b), and $d_1$ takes (a,b) to a+b. Observe that 
\begin{align*}
(C\otimes C)_n \cong \begin{cases} (k \cong k)\otimes_k k & \mbox{ if } n=0\\
(k\otimes_k (k\oplus k)) \oplus ((k\oplus k)\otimes k) \cong k^{\oplus 4} & \mbox{ if } n=1\\
 (k \otimes_k (k\oplus k)) \oplus ((k\oplus k)\otimes_k (k\oplus k)) \oplus ((k\oplus k)\otimes_k k) \cong k^8 & \mbox{ if } n=2\\
\end{cases}
\end{align*}

The differential $d_1$ takes $(u \otimes (a,b),(c,d)\otimes w)$ to $ua+ub+cw+dw$, while $d_2$ takes $(u\otimes (a,b),(e,f)\otimes (s,t),(c,d)\otimes w)$ to $(ua+ub+es+et+fs+ft,ua+ub+es+et+fs+ft,wc+wd+es+et+fs+ft,wc+wd+es+et+fs+ft)$. 

Now consider $(C\otimes C)/\Sigma_2$, where the $\Sigma_2$ action is induced by swapping $C_p$ and $C_q$ in the formula $\bigoplus_{n=p+q} C_p\otimes_k C_q$. The action on the degree 1 part swaps the first two coordinates and swaps the second two coordinates. The differential $d_1$ is an epimorphism, while im$(d_2)$ lies in the the $\Sigma_2$-invariant subspace, hence goes to zero when we pass to coinvariants. It follows that $H_1(C\otimes C/\Sigma_2) \cong k$.

This example demonstrates that symmetric operads in Ch$(\F_2)$ is not a full model structure, since pushouts of trivial cofibrations need not be weak equivalences. This example does not rule out the transferred semi-model structure because Com is not cofibrant, and only pushouts of trivial cofibrations into cofibrant objects need to be again trivial cofibrations for semi-model categories.
\end{example}

\section{Algebras in a localized category}
\label{sec:alg-in-loc}

Suppose $\M$ is a cofibrantly generated model category, $\cat{C}$ is a class of morphisms such that $L_\cat{C}(\M)$ exists, and $T$ is a (finitary) monad on $\M$ such that $\algtm$ inherits a transferred model structure from $\M$. 

\begin{defn} \label{defn:local-proj-model}
A \textit{local projective (semi)-model structure} on $\algtm$ is the lifting of $\lcm$ along the forgetful functor $U:\algtm \to \M.$  We will denote this lifting $\algtlcm$.
\end{defn}

\begin{defn} We will call an algebra $Z\in \algtm$ \textit{local} if $U(Z)$ is a local object in $\M.$ 
\end{defn}

\begin{lemma}  $Z\in \algtm$ is a local object in $\algtm$ with respect to $T(\cat{C})$ if and only if $Z$ is a local algebra. \end{lemma}

\begin{proof}
An algebra $Z$ is a local  object if and only if for any $f\in T(\cat{C}), f:X\to Y$ between cofibrant algebras, the map 
$$Map_{\algtm}(f,Z): Map_{\algtm}(X,Z) \leftarrow Map_{\algtm}(Y,Z) $$ 
is a weak equivalence of simplicial sets. Since $X$ and $Y$ are cofibrant algebras, the function complexes  $Map_{\algtm}(X,Z), Map_{\algtm}(Y,Z)$ can be constructed as   simplicial sets
$\algtm(X,Z_*)$ and $ \algtm(Y,Z_*)$ where $Z_*$ is a simplicial resolution of $Z$ in the category of algebras. Since $U$ preserves limits, fibrations, and weak equivalences, $U(Z_*)$ is a simplicial resolution of $U(Z).$

Finally, since $f = T(g)$ for $g: V\to W$, by adjointness we have that $Map_{\algtm}(f,Z)$ is a weak equivalence  if and only if the map
$$Map_{\M}(g,U(Z)): Map_{\M}(V,U(Z)) \leftarrow Map_{\M}(W,U(Z)) $$ 
is a weak equivalence. This occurs if and only if $U(Z)$ is a  local object in $\M$, as required. 
\end{proof}

\begin{theorem}\label{thm:existence-implies-locAlg=algLoc}
Suppose the transferred model structures $\algtm$ and Alg$_T(L_\cat{C}(\M))$ exist. Then the localization $\ltcalgtm$ exists and coincides with $\algtlcm$. Furthermore, if $\algtm$ and $\algtlcm$ exist as semi-model categories then the local projective semi-model structure $\ltcalgtm$ exists and coincides with the semi-model structure $\algtlcm$.
\end{theorem}

This result means that, of the two ways of going around the following diagram to study local algebras, the ability to go counterclockwise (localize then transfer) implies the ability to go clockwise (transfer then localize). The opposite is not true: there are examples where one can go clockwise, but not counterclockwise (see Remark \ref{remark:semi-loc} below).

\begin{align*}
\xymatrix{
\algtm \ar@{|->}[r]^(.3){\ltc} & \algtlcm = \ltcalgtm \\
\M \ar@{|->}[u] \ar@{|->}[r]_\lc & \lcm \ar@{|->}[u]}
\end{align*}

The importance of semi-model structures is that transfers to categories of algebras often only result in semi-model structures (see \cite{white-yau}), especially when one is transferring from $\lcm$ where one has lost control over the trivial cofibrations. We provide an example at the end of this section for a semi-model structure on $\algtm$ that is provably not a model structure.

\begin{proof}
We first focus on the situation of model structures, delaying discussion of semi-model structures until the end of the proof. We first show that the identity functor $Id:\algtm \rightarrow \algtlcm$ is a left Quillen functor. It is sufficient to prove that  its inverse $Id^{-1}:\algtlcm \rightarrow \algtm$ maps (trivial) fibrations to (trivial) fibrations. A fibration in $\algtlcm$ is a morphism $f:X\rightarrow Y$ such that $U(f)$ is a fibration in $\lcm$. Hence, $Id^{-1}(U(f))$ is also a fibration in $\M$ because $Id^{-1}:\lcm^{loc}\rightarrow \M$ is a right Quilen functor. Therefore, $Id^{-1}(f)$ is also a fibration in $\algtm$ because $\algtm$ carries a transferred model structure. The same argument applies for trivial fibrations.

As a consequence, the class of cofibrations in $\algtm$ coincides with the class of cofibrations in $\algtlcm$. In addition, an adjunction argument shows that the fibrant objects in $\algtlcm$ are exactly fibrant $T(\cat{C})$-local algebras. 

Finally, $f:X\rightarrow Y$ is a $T(\cat{C})$-local equivalence in $\algtm$ if and only if it induces a weak equivalence between simplicial sets:
$$Map_{\algtm}(X,Z)= \algtm(X^*,Z)\leftarrow \algtm(Y^*,Z) = Map_{\algtm}(Y,Z)$$
for all fibrant local objects $Z\in \algtm$, where $X^*,Y^*$ are cosimplicial resolutions of $X$ and $Y$ \cite{barwickSemi} (Scholium 3.64). But any cosimplicial resolution in $\algtm$ is also a cosimplicial resolution in $\algtlcm$ because they have the same class of cofibrations, and because any weak equivalence in $\algtm$ is also a weak equivalence in $\algtlcm.$ Hence, we have
$$Map_{\algtm}(X,Z) = \algtm(X^*,Z) = \algtlcm(X^*,Z) = Map_{\algtlcm}(X,Z)$$
for any local fibrant object  $Z$ since it is also a fibrant object in $\algtlcm.$ So, $f$ induces a weak equivalence
  $$Map_{\algtlcm}(X,Z)\leftarrow Map_{\algtlcm}(Y,Z)$$
  for any fibrant object $Z$ in $\algtlcm$. Hence, $f$ is a weak equivalence in $\algtlcm$, as required. 

For semi-model categories, precisely the same proof works. As the first paragraph focuses entirely on fibrations and trivial fibrations, it applies verbatim to semi-model structures and proves that the classes of fibrations in these two semi-model structures on $\algtm$ coincide. The second paragraph proves that cofibrations coincide, by lifting. This works for semi-model categories as well.
Lastly, to prove that weak equivalences coincide, note that Scholium 3.64 in \cite{barwickSemi} is written to work for semi-model categories (there called left model categories), and that the machinery to carry out cosimplicial resolution for cofibrant objects is also provided (indeed, goes back to \cite{spitzweck-thesis}). Thus, maps $f:X\to Y$ between cofibrant objects are weak equivalences in $\algtlcm$. For a general $T(\cat{C})$-local equivalence $f:A\to B$, replace $f$ by its cofibrant replacement $Qf$ in $\algtlcm$. Note that the maps $QA\to A$ and $QB\to B$ are trivial fibrations in $\algtlcm$, hence in $\lcm$, hence in $\M$. So the two out of three property for $T(\cat{C})$-local equivalences (which holds independently of the existence of a semi-model structure on $\ltcalgtm$) shows that $Qf$ is a $T(\cat{C})$-local equivalence. The argument just given proves that $Qf$ is a weak equivalence in $\algtlcm$, as are the maps $QA\to A$ and $QB\to B$. Thus, by the two out of three property, so is $f$, completing our proof.
\end{proof}

The hypotheses of this Theorem are satisfied for the free monoid monad if $L_\cat{C}(\M)$ satisfies the pushout product and monoid axioms. Conditions on $\M$ and $\cat{C}$ are provided in \cite{white-localization} so that this occurs. For $T=\Sym$ one must also prove that $L_\cat{C}(\M)$ satisfies the commutative monoid axiom, but again conditions for this to occur are given in \cite{white-localization}.

We now characterize when this lifted model structure exists. When we say $U$ preserves local equivalences we mean $U$ takes $T(\cat{C})$-local equivalences to $\cat{C}$-local equivalences.

\begin{theorem}\label{thm:batanin-big-iff} 
Let $\M$ be a $\cat{K}$-compactly generated model category where $\cat{K}$ is a saturated class in $\M.$ Let $T$ be a $\cat{K}$-admissible monad on $\M.$ Assume the localization $\lcm$ is a $\cat{K}$-compactly generated monoidal model category and, moreover, the domains of generating trivial cofibrations in $\lcm$ are cofibrant (for example, if the domains are cofibrant in $\M$). Assume also that the projective structure on $\algtm$ is left proper.  

Then the transferred model structure $\algtlcm$ exists if and only if  $U$ preserves local weak equivalences.
\end{theorem}

Note that conditions are given in \cite{white-localization} so that $L_\cat{C}(\M)$ satisfies the hypothesis of being a $\cat{K}$-compactly generated monoidal model category. Furthermore, Corollary 4.15 in \cite{white-localization} provides conditions under which cofibrancy for the domains of maps in $J$ can be deduced.

For clarity in the proof, we use $U^{loc}$ to denote the forgetful functor from $\algtlcm$ to $\lcm$. As a functor, this is of course the same as $U$ from $\algtm$ to $\M$.

\begin{proof}
If the transfer from $\lcm$ to $\algtlcm$ exists, then by Theorem \ref{thm:existence-implies-locAlg=algLoc} the transferred model structure coincides with the lifted localization. So, $U^{loc}$ preserves weak equivalences. Hence, $U$ preserves local weak equivalences.

Let us prove the converse. The hypotheses of the theorem allow us to use the  machinery from \cite{batanin-berger} to prove that the transfer exists if we can prove that $T$ is $\cat{K}$-admissible for $\lcm$. For this we need to analyze a pushout created from a generating trivial cofibration $f:K\rightarrow L$ in $\lcm$, and a morphism of algebras $g: F(K)\rightarrow X$:

\begin{equation}\label{pushout} 
\xymatrix{F(K) \ar[r]^{F(f)} \ar[d]_g & F(L) \ar[d] \\ X \ar[r]^{\tilde{f}} & P}
\end{equation}

We must prove that underlying morphism of $\tilde{f}$ belongs to $\cat{K}$ and is a local weak equivalence. Observe that $U(\tilde{f})$ is in $K$ because $T$ is $\cat{K}$-admissible. We must prove that it is a local weak equivalence. Since the transferred model structure on Alg$(T)$ is left proper (from \cite{batanin-berger}), the pushout (\ref{pushout}) is actually a homotopy pushout in Alg$(T).$ Let $Z$ be a local fibrant algebra, so that $U(Z)$ is a local fibrant object in $\lcm.$ Apply the simplicial mapping space functor $Map_{Alg(T)}(-,Z)$ to the pushout (\ref{pushout}). Since $F(K)$ and $F(L)$ are cofibrant, $Map_{Alg(T)}(-,Z)$ can be constructed as a levelwise $Hom$ functor $Alg(T)(-,Z_*),$ where $Z_*$ is a simplicial resolution of $Z.$ Note that $U(Z_*)$ is a simplicial resolution of $U(Z)$ in $\lcm$, so adjunction yields a homotopy pullback in simplicial sets:

\[
\xymatrix{
Map_\M(K,U(Z)) & & Map_\M(L,U(Z)) \ar[ll]_{Map_\M(f,U(Z))} \\ 
Map_{Alg(T)}(X,Z) \ar[u] && Map_{Alg(T)}(P,Z) \ar[u] \ar[ll]_{Map_{Alg(T)}(\tilde{f},Z)}
}
\]

Note that $Map_{\lcm}(f,U(Z))$ is a trivial fibration since $f$ is a local trivial cofibration and $U(Z)$ is a fibrant local object. Therefore, $Map_{Alg(T)}(\tilde{f},Z)$ is a trivial fibration for all $Z$, so $\tilde{f}$ is a local equivalence in $Alg(T).$ Since $U$ preserves local weak equivalences, $U(f)$ is a local equivalence and we have completed the proof.
\end{proof}

\begin{corollary} \label{cor:batanin-big-iff-semi}
Let $\M$ be a left proper, $\cat{K}$-compactly generated monoidal model category, for some saturated class of morphisms $\cat{K}$. Let $T$ be a $\cat{K}$-semi-admissible monad (so there is a transferred semi-model structure $\algtm$). Assume $\lcm$ exists as a monoidal model category and that the domains of its generating trivial cofibrations are cofibrant. Then the transferred semi-model structure on $\algtlcm$ exists if and only if $U$ preserves local weak equivalences. In addition, if cofibrant algebras are relatively cofibrant then we obtain on $\algtlcm$ a transferred semi-model structure over $\lcm$.
\end{corollary}

Note that conditions such that cofibrant algebras are relatively cofibrant are given in \cite{white-yau} (Theorems 6.2.3 and 6.3.1).

\begin{proof}
We proceed as in the proof of the theorem, but assume that the algebra $X$ is a cofibrant algebra. We form the pushout square 
\vspace{12mm}
\begin{equation}\label{cor:semi-pushout} 
\xymatrix{
F(K) \ar[r]^{F(f)} \ar[d]_g & F(L) \ar[d] \\
X \ar[r]^{\tilde{f}} & P
}
\end{equation}

and we note that it is a homotopy pushout square even if $\algtm$ is not left proper, because all objects are cofibrant in $\algtm$ and $F(f)$ is a cofibration of algebras. The rest of the proof follows precisely as in the theorem, using that $F(K)$ and $F(L)$ are cofibrant algebras. If cofibrant algebras are relatively cofibrant, then we can do the same proof starting with an algebra $X$ cofibrant in $\M$ and using that $F(K)$ and $F(L)$ are also cofibrant in $\M$ to deduce the pushout square is a homotopy pushout.
\end{proof}

\section{Localizing a category of algebras}
\label{sec:loc-in-alg-carato}

In the previous section we provided conditions so that we could first localize $\M$ at a class of maps $\cat{C}$ and then transfer the local (semi)-model structure to the category of $T$-algebras. An alternative way to study local $T$-algebras is to lift the localization $L_\cat{C}$ to a localization $\ltc$ on the category of algebras (where $T(\cat{C})$ are the free $T$-algebra maps on $\cat{C}$). Conditions for such lifts to exist have been found in \cite{carato}, along with an extensive discussion of colocalization and a comparison between $L_\cat{C}$ and $L_{U(T(\cat{C}))}$ with applications to classical localizations for spaces, abelian groups, and spectra. We will now summarize the main results of \cite{carato} that are related to our results on localizations, and then we deduce new results regarding preservation of $T$-algebra structure by $L_\cat{C}$ in the next section, building both on our work above and on \cite{carato}.

In \cite{carato}, the localizations under consideration are \textit{homotopical localizations}. Usually, such a setting would be a slight generalization of the notion of a left Bousfield localization, avoiding the need to assume there is a \textit{set} of maps being inverted. However, in \cite{carato} the focus is on homotopical localizations that come from left Bousfield localizations $L_f$, so the setting matches ours. A key concern of \cite{carato} is determining when a left Bousfield localization $L_\cat{C}$ on $\M$ lifts to a localization of algebras in the following sense:

\begin{defn} \label{defn:carato-lifting-defn}
We say that \textit{$L_\cat{C}$ lifts to the homotopy category of $T$-algebras} if there is an endofunctor $L^T$ on $\hoalgtm$ together with a natural isomorphism $h:L_\cat{C} U \to UL^T$ in $\ho(\M)$ such that $h \; \circ \; lU = Ul^T$ in $\ho(\M)$ where $l_X:X\to L_\cat{C}(X)$ and $l^T_E:E\to L^T(E)$ are the localization maps. 
\end{defn}

The main results from \cite{carato} that we will be interested in regard when forgetful functors preserve local equivalences (as in Theorem \ref{thm:existence-implies-locAlg=algLoc}) and when localizations lift to the category of algebras. The following two results appear in Section 7 of \cite{carato}, but have been reworded to match our situation.

\begin{lemma} \label{lemma:carato-beta-iff-U-preserves}
For a localization $L^T$ on Ho($\algtm$) and a localization $L$ on Ho($\M$), the functors $UL^T$ and $LU$ are naturally isomorphic if and only if $U:\hoalgtm \to \Ho(\M)$ preserves local objects and equivalences.
\end{lemma}

The following is a uniqueness theorem for the lift, proving that if the lift exists then it must be the lift we expect to exist, namely $L_{T(\cat{C})}$. Note, however, that $L_{T(\cat{C})}$ could exist and not be a lift of $L_\cat{C}$. An example is provided in Example \ref{example:postnikov}. Conditions to force $L_{T(\cat{C})}$ to be a lift of $L_\cat{C}$ are given in \cite{gutierrez-rondigs-spitzweck-ostvaer-colocalizations}. We have weakened the hypothesis in \cite{carato} from requiring full transferred model structures to only requiring semi-model structures:

\begin{theorem} \label{thm:carato-if-lift-exists}
Suppose $\algtm$ has a transferred semi-model structure, that $\ltcalgtm$ exists as a semi-model category, and that $L_\cat{C}$ lifts to Ho$(\algtm)$. Then:
\begin{enumerate}
\item $T$ preserves $\cat{C}$-local equivalences between cofibrant objects, 
\item There is a natural isomorphism $\beta_X:L_\cat{C} U(X) \cong U\ltc(X)$ in Ho$(\M)$ for all algebras $X$, and 
\item $U$ preserves and reflects local equivalences, trivial fibrations, and fibrant objects.
\end{enumerate}
\end{theorem}

The reason we can weaken the hypothesis to only requiring a semi-model structure is that the proof in \cite{carato} only ever works on the homotopy level or on the subcategory of cofibrant objects. The argument required to define $\beta_X$ (recalled in the next section) requires a lift in $\lcm$, not in a category of algebras. We will see some consequences of this generalization in Section \ref{sec:preservation}, and the generalization allows us to apply this theory to the many situations where only semi-model structures are known.

\section{Preservation of algebras by localization}
\label{sec:preservation}

The following definition has appeared in \cite{white-localization}, where $\M$ is a model category and $\cat{C}$ is a set of maps in $\M$. Write $L_\cat{C}$ for the composition of derived functors $\Ho(\M)\to \Ho(L_\cat{C}(\M)) \to \Ho(\M)$, i.e. $E\to L_\cat{C}(E)$ is the unit map of the adjunction $\Ho(\M) \leftrightarrows \Ho(L_\cat{C}(\M))$.

\begin{defn} \label{defn:preservation-white}
$L_\cat{C}$ is said to \textit{preserve $T$-algebras} if 

\begin{enumerate}
\item When $E$ is a $T$-algebra there is some $T$-algebra $\tilde{E}$ which is weakly equivalent in $\M$ to $L_\cat{C}(E)$.

\item In addition, when $E$ is a cofibrant $T$-algebra, then there is a choice of $\tilde{E}$ in $\algtm$ with $U(\tilde{E})$ local in $\M$, there is a $T$-algebra homomorphism $r_E:E\to \tilde{E}$ that lifts the localization map $l_E:E\to L_\cat{C}(E)$ up to homotopy, and there is a weak equivalence $\beta_E:L_\cat{C}(UE)\to U\tilde{E}$ such that $\beta_E \circ l_{UE} \cong Ur_E$ in $\ho(\M)$.
\end{enumerate}
\end{defn}

Observe that, when $L_\cat{C}$ lifts to the homotopy category of $T$-algebras in the sense of Definition \ref{defn:carato-lifting-defn}, it implies the preservation of Definition \ref{defn:preservation-white} on the homotopy category level, but does not necessarily imply there is an actual \textit{map} from $E$ to $\tilde{E}$ in $\algtm$. However, Definition \ref{defn:preservation-white} does imply that $L_\cat{C}$ lifts to the homotopy category of $T$-algebras, and naturality of $\beta$ is deduced as part of Theorem \ref{thm:omnibus}.

The following is proven in \cite{white-localization}, and provides a host of examples where preservation occurs, especially for algebras over $\Sigma$-cofibrant colored operads, over entrywise cofibrant colored operads \cite{white-yau}, and for commutative monoids.

\begin{theorem} \label{thm:white-preservation}
If $\algtm$ and $\algtlcm$ have transferred semi-model structures then $L_\cat{C}$ preserves $T$-algebras.
\end{theorem}

\begin{proof}
This has already been proven in \cite{white-localization} and \cite{white-yau}, but for the sake of being self-contained we recall the main points of the proof. We must verify the statements in Definition \ref{defn:preservation-white}. Let $E$ be a $T$-algebra, and define $\tilde{E}$ to be $R_{\cat{C},T} Q_T E$ where $R_{\cat{C},T}$ (resp. $Q_T$) is fibrant replacement (resp. cofibrant replacement) in $\algtlcm$. This $\tilde{E}$ is a $T$-algebra because we have taken replacements in the category of $T$-algebras. That $\tilde{E}$ is weakly equivalent to $L_\cat{C}(E)$ is proven in \cite{white-localization} and \cite{white-yau} by constructing local weak equivalences $L_\cat{C}(E) \simeq R_\cat{C} QE \to R_\cat{C} Q_T E \to R_{\cat{C},T} Q_T E$ and then observing that a local weak equivalence between local objects (using that $\algtlcm$ is transferred from $\lcm$) is a weak equivalence. Observe that $U(\tilde{E})$ is local because the semi-model structure on $\algtlcm$ is transferred. When $E$ is already cofibrant, the map $r_E:E\to \tilde{E}$ is just the fibrant replacement map $R_{\cat{C},T}$, and the comparison $\beta_E$ is the following lift in $\lcm$:

\begin{align*}
\xymatrix{UE \ar@{^(->}[d]_{\simeq_\cat{C}} \ar[r] & U\tilde{E} \ar@{->>}[d] \\
L_\cat{C}(UE) \ar[r] \ar@{..>}[ur]_{\beta_E}& \ast}
\end{align*}
\end{proof}

We will see in Theorem \ref{thm:omnibus} that, under mild hypotheses on $\M$ and $T$, if $L_\cat{C}$ preserves $T$-algebras, then the transferred semi-model structure $\algtlcm$ exists. Furthermore, this implies the localization $L_\cat{C}$ lifts, as we now show.

\begin{proposition} \label{prop:preservation-notions-agree}
Suppose $\algtm$ and $\algtlcm$ have transferred semi-model structures. Then $L_\cat{C}$ lifts to the homotopy category of $T$-algebras in the sense of Definition \ref{defn:carato-lifting-defn}.
\end{proposition}

\begin{proof}
Define the localization $L^T$ to be (the image in $\hoalgtm$ of) $R_{\cat{C},T}$. The natural isomorphism $h$ will be the image in $\ho(\M)$ of $\beta$. We construct $\beta$ via lifting in $\lcm$, using Theorem \ref{thm:existence-implies-locAlg=algLoc} to realize that $\algtlcm = \ltcalgtm$ and so $R_{\cat{C},T} = R_{T(\cat{C})}$:

\begin{align*}
\xymatrix{UE \ar@{^(->}[d]_{\simeq_\cat{C}} \ar[r] & UR_{T(\cat{C})}E \ar@{->>}[d] \\
R_\cat{C} UE \ar[r] \ar@{..>}[ur]_{\beta_E}& \ast}
\end{align*}

This lift demonstrates immediately that $\beta \circ lU = Ul^T$ in $\ho(\M)$. Note that the top horizontal map is a $\cat{C}$-local equivalence because $\algtlcm$ has the transferred model structure, so the lift is a $\cat{C}$-local equivalence by the two out of three property. Note that the domain and codomain of the lift are $\cat{C}$-local objects (again using that the fibrations are transferred from $\lcm$). Hence, the lift is a weak equivalence in $\M$. In addition, the lift is unique in $\ho(\M)$ by the universal property of localization, since any other lift would necessarily be a weak equivalence in $\M$ between the same two $\cat{C}$-local objects. Finally, the lift is natural in $\ho(\M)$ because if we began with a map $E\to F$ and constructed this lift on its domain and codomain then we could in addition construct a homotopy unique lift from $R_\cat{C} UE$ to $UR_{T(\cat{C})}F$, and so uniqueness tells us the relevant naturality square commutes in $\ho(\M)$.
\end{proof}

\begin{proposition} \label{prop:strong-pres-implies-U-preserves}
If $L_\cat{C}$ preserves $T$-algebras then $U$ preserves local equivalences.
\end{proposition}

\begin{proof}
Let $E$ be a $T$-algebra, cofibrant in $\M$. Consider the following diagram, guaranteed to exist because $L_\cat{C}$ preserves $T$-algebras:

\begin{align*}
\xymatrix{U(E) \ar[r]^{Ur_E} \ar[d]_{l_E} & U(\tilde{E})\\
L_\cat{C}(U(E)) \ar[ur]_{\beta_E} & \\}
\end{align*}

Since $l_E:U(E) \to L_\cat{C}(U(E))$ is a local equivalence and $\beta_E$ is a (local) equivalence, $Ur_E$ is a local equivalence. So $U$ preserves local equivalences of the form $r_E$.

Suppose $f:E\to F$ is a $T(\cat{C})$-local equivalence in $\algtm$. Consider the diagram

\begin{align*}
\xymatrix{
U(E) \ar[rrr]^{U(f)} \ar[dr] \ar[dd]_{U(r)} & & & U(F) \ar[dl] \ar[dd]^{U(r)}\\
& L_\cat{C} U(E) \ar[r]^{L_\cat{C}(U(f))} \ar[dl] & L_\cat{C} U(F) \ar[dr] & \\
\tilde{E} \ar@{..>}[rrr] &&& \tilde{F}}
\end{align*}

We proved above that the outside vertical maps are local equivalences. The maps in the lower trapezoid are all weak equivalences in $\M$ because they are local equivalences between local objects. Thus, in $\ho(\M)$, the dotted arrow exists and is an isomorphism. It follows that $U(f)$ is a local equivalence because its localization is an isomorphism in $\ho(\M)$.

\end{proof}

We turn now to the consequences we can deduce when preservation occurs. Recall that semi-admissibility is a weak hypothesis often satisfied in practice by monads encoding colored operads.

\begin{theorem}
Assume $T$ is $\cat{K}$-semi-admissible, $\M$ is $\cat{K}$-compactly generated, and $\lcm$ is monoidal, $\cat{K}$-compactly generated, and has cofibrant domains of the generating trivial cofibrations, for some saturated class of morphisms $\cat{K}$. If $L_\cat{C}$ lifts to the homotopy category of $T$-algebras (as in Definition \ref{defn:carato-lifting-defn}) then Alg$_T (L_\cat{C}(\M))$ has a transferred semi-model structure, which is a full model structure if $T$ is $\cat{K}$-admissible and if Alg$_T(\M)$ is left proper. In either case, the lifted homotopy localization $L^T$ is induced by a left Bousfield localization $\ltcalgtm$.
\end{theorem}

\begin{proof}
Lemma \ref{lemma:carato-beta-iff-U-preserves} proves that $U:\hoalgtm \to \ho(\M)$ preserves local equivalences. Note that $U:\algtm \to \M$ also preserves local equivalences.
Theorem \ref{thm:batanin-big-iff} (resp. Corollary \ref{cor:batanin-big-iff-semi}) implies that $\algtlcm$ exists. Theorem \ref{thm:existence-implies-locAlg=algLoc} implies $\ltcalgtm$ exists and coincides with $\algtlcm$. The homotopy uniqueness of $\beta$ in Theorem \ref{thm:carato-if-lift-exists} implies $L_{T(\cat{C})}$ is a lift of $L_\cat{C}$ to the model category level, i.e. agrees with $L^T$.
\end{proof}

Note that the converse to this theorem also holds, i.e. we can deduce that $L_\cat{C}$ lifts to the homotopy category of $T$-algebras if we know $\algtlcm$ exists as a semi-model category, using Proposition \ref{prop:preservation-notions-agree}.

We are finally ready for our omnibus theorem relating the notions considered in this paper and \cite{carato}.

\begin{theorem} \label{thm:omnibus}
Let $\cat{K}$ be a saturated class of morphisms. Assume $T$ is $\cat{K}$-semi-admissible, $\M$ is $\cat{K}$-compactly generated, and $\lcm$ is $\cat{K}$-compactly generated with cofibrant domains of the generating trivial cofibrations. The following are equivalent:
\begin{enumerate}
\item $L_\cat{C}$ lifts to a left Bousfield localization $\ltc$ of semi-model categories on $\algtm$.
\item $U:\algtm \to \M$ preserves local equivalences.
\item $\algtlcm$ has a transferred semi-model structure.
\item $L_\cat{C}$ preserves $T$-algebras.
\end{enumerate}
Furthermore, any of the above implies
\begin{enumerate}
\item[(5)] $T$ preserves $\cat{C}$-local equivalences between cofibrant objects.
\end{enumerate}
\end{theorem}

In Theorem \ref{thm:implies-reflects-local}, we provide a partial converse, i.e. we determine what can be said when only (5) is known.

\begin{proof}
That (1) implies (2) is part of Lemma \ref{lemma:carato-beta-iff-U-preserves}. That (2) is equivalent to (3) is Corollary \ref{cor:batanin-big-iff-semi}, and that (3) implies (1) is Theorem \ref{thm:existence-implies-locAlg=algLoc}. That (3) implies (4) is Theorem \ref{thm:white-preservation}. That (4) implies (2) is Proposition \ref{prop:strong-pres-implies-U-preserves}. 
That (1) implies (5) is part of Theorem \ref{thm:carato-if-lift-exists}. 
\end{proof}

\begin{remark} \label{remark:semi-loc}
The second author has shown that, for locally presentable cofibrantly generated semi-model categories, $\ltcalgtm$ exists as a semi-model category whenever $\algtm$ does \cite{white-bous-loc-semi}. However, this does not mean that preservation comes for free, because one does not know that the resulting localization lifts $L_\cat{C}$ unless one also knows that $U$ preserves local equivalences, as the next example shows.
\end{remark}

\begin{example} \label{example:postnikov}
Let $\M$ be the category of symmetric spectra, with the stable model structure \cite{hovey-shipley-smith}. Recall that the $n^{th}$ Postnikov section functor $P_n$ is the Bousfield localization $L_f$ corresponding to the map $\Sigma^\infty (\Sigma f)$ where $f:S^n \to *$. Recall from \cite{cgmv} that $P_{-1}$ does not preserve monoids, because $S\wedge P_{-1}R \to P_{-1}R \wedge P_{-1}R \to P_{-1}R$ is not a homotopy equivalence as it would have to be for $P_{-1}R$ to be a homotopy ring. Thus, none of the 5 conclusions of Theorem \ref{thm:omnibus} can be satisfied. However, there is a semi-model categorical left Bousfield localization $L_{Ass(P_{-1})}$ on the category of ring spectra, by \cite{white-bous-loc-semi}. This localization cannot be a lift of $P_{-1}$, because it would violate Theorem \ref{thm:omnibus}.
\end{example}

\begin{example} \label{example:comm-mon-sym}
In \cite{white-commutative-monoids}, the second author introduced the \textit{commutative monoid axiom} and proved that it implies the existence of a transferred model structure on commutative monoids. Let $\Sym$ denote the free commutative monoid functor. In \cite{white-localization}, it was proven that if a monoidal left Bousfield localization $L_\cat{C}$ has the property that $\Sym(\cat{C})$ is contained in the $\cat{C}$-local equivalences, then $L_\cat{C}(\M)$ inherits the commutative monoid axiom from $\M$, and hence $L_\cat{C}$ preserve commutative monoids.  Theorem \ref{thm:omnibus} implies the converse, i.e. if $L_\cat{C}$ preserves commutative monoids then $\Sym$ must preserve $\cat{C}$-local equivalences. This answers a question posed to the second author by Nito Kitchloo.
\end{example}

\begin{remark}
The main result of \cite{gutierrez-rondigs-spitzweck-ostvaer-colocalizations} complements Theorem \ref{thm:omnibus}. The result states that, if $\M$ is a simplicial monoidal model category, if $T$ comes from a colored operad $O$, if $\algtm$ inherits a transferred model structure from $\M$, if $L_\cat{C}$ is a monoidal left Bousfield localization and $O(c_1,\dots,c_n;c)\otimes -$ preserves $\cat{C}$-local equivalences, and if the model category $\ltcalgtm$ exists, then $U$ preserves local equivalences (and hence, all 5 of the statements in Theorem \ref{thm:omnibus} are true). The authors wonder if the results of this paper can be made to work with only semi-model structures. If so, the existence of $\ltcalgtm$ appears to occur more frequently than the existence of $\algtlcm$, so this result might give easier to check conditions such that the 5 equivalent statements in Theorem \ref{thm:omnibus} hold. Note that \cite{gutierrez-rondigs-spitzweck-ostvaer-colocalizations} also considers colocalizations.
\end{remark}

We conclude the paper with a partial converse to the last implication in Theorem \ref{thm:omnibus}, i.e. we state what can be deduced from knowing that a monad $T:\M \to \M$ preserves $\cat{C}$-local equivalences. Recall that, for any monad $(T,\mu,\epsilon)$ on $\M$ and any $X\in \M$, there is an augmented cosimplicial object defined as follows:
\begin{equation*}
\xymatrix{T^*(X) := X \ar[r]  & T(X) \ar@<0.5ex>[r]^{T \epsilon} \ar@<-0.5ex>[r]_{\epsilon_T} & \ar[l] T^2(X) \ar@<1ex>[r]\ar[r]\ar@<-1ex>[r] & T^3(X)  \ar@<0.5ex>[l] \ar@<-0.5ex>[l] \ar@<1ex>[r]\ar@<0.35ex>[r]\ar@<-0.35ex>[r]\ar@<-1ex>[r] & \dots \ar@<1ex>[l]\ar[l]\ar@<-1ex>[l]}
\end{equation*}
The maps going left are multiplication $\mu$ on $T$. The maps going right are induced by the unit $\epsilon$. For our partial converse to Theorem \ref{thm:omnibus}, we will need the following definition, which goes back to \cite{BatA}.

\begin{definition}
We will say that the monad $(T,\mu,\epsilon)$ on $\M$ is \textit{pointwise Reedy cofibrant} if, for any cofibrant $X\in \M$, the augmented cosimplicial object $T^*(X)$ is Reedy cofibrant.
\end{definition}

Many monads we come across in practice are pointwise Reedy cofibrant. For example, if $O$ is any nonsymmetric operad where $I \to O(1)$ is cofibration (where $I$ is the unit of $\M$), then the free $O$-algebra functor is pointwise Reedy cofibrant \cite{BatA}. If $T$ is a polynomial monad and the unit of $\M$ is cofibrant then $T$ is pointwise Reedy cofibrant \cite{batanin-berger}. Work in progress of Mark Johnson and Donald Yau provides even more examples. With this definition in hand, we are ready for our partial converse to Theorem \ref{thm:omnibus}. We do not expect a full converse to be true in general.

\begin{theorem}\label{thm:implies-reflects-local}  
Suppose $\algtm$ admits a transferred semi-model structure and that:
\begin{itemize}
\item $T$ is a pointwise Reedy cofibrant monad, 
\item $T$ preserves local weak equivalences between cofibrant objects, and
\item $U$ sends cofibrant algebras to cofibrant objects.
\end{itemize}
Then $U$ reflects local weak equivalences between cofibrant algebras.
\end{theorem}

\begin{proof} Suppose that $U$ preserves cofibrant objects.
Let $f:X \to Y$ be a morphism between cofibrant algebras such that $U(f)$ is a local weak equivalence. We must show $f$ is a local weak equivalence. Let $Z$ be a local fibrant algebra. We need to show that $Map_{Alg(T)}(X,Z)\leftarrow Map_{Alg(T)}(Y,Z)$ is a weak equivalence. Since $X$ is cofibrant, the mapping space $Map_{Alg(T)}(X,Z)$ can be constructed as $Alg(T)(X,Z_*)$ where $Z_*$ is a simplicial resolution of $Z$, again using \cite{barwickSemi} [3.64]. Observe that $Z_n$ is a local algebra for each $Z$ because $Z$ is local. We have a classical simplicial bar resolution for $X$, defined as $X\leftarrow T^*(X)$ in $\algtm$, in which all terms $T^n(X)$ are cofibrant in $\M$ (since $U(X)$ is cofibrant).  After applying $Map_{Alg(T)}(-,Z)$ we have, therefore, 
an augmented cosimplicial simplicial set:

$$ Map_{Alg(T)}(X,Z) = Alg(T)(X,Z_*) \to  Alg(T)(T^*(X),Z_*)$$ 

Then the cosimplicial simplicial set $Alg(T)(T^*(X),Z_*) = \M(T^{*-1}(U(X)),U(Z_*))$ is Reedy fibrant by the Homotopy Lifting Extension Theorem \cite{hirschhorn}[Corollary 16.5.14] applied in $\M$, using that $T$ is pointwise Reedy cofibrant. Moreover, this cosimplicial simplicial set has an extra degeneracy (precomposing with $X\to T(X)$) which shows that the augmentaion 

$$ Map_{Alg(T)}(X,Z) = Alg(T)(X,Z_*) \to  \M(T^{*-1}(U(X)),U(Z_*))$$ 

induces a deformation retraction (Lemma 2.1, \cite{batanin-prohomotopy}) of fibrant simplicial sets

$$ Map_{Alg(T)}(X,Z)  \to  Tot(\M(T^{*-1}(U(X)),U(Z_*))).$$ 

Hence, it suffices to show that $f$ induces a weak equivalence

$$ \M(T^{k}(U(X)),U(Z_*)) \stackrel{g}{\leftarrow} \M(T^{k}(U(Y)),U(Z_*)) $$

for all $k\ge 0.$ 
As we have assumed $U(f)$ is a local weak equivalence between cofibrant algebras, and that such local weak equivalences are preserved by $T$, the map $g$ is a weak equivalence, since all $Z_n$ are local fibrant algebras. Hence, $f$ is a local weak equivalence as required.
\end{proof}

We conclude with an example demonstrating the value of the generality of our approach (working with monads that need not come from operads), and an application of Theorem \ref{thm:implies-reflects-local}. 

\begin{example}
Let $\cat{V}$ be a monoidal model category and let $\cat{O}$ be any symmetric colored operad in $\cat{V}$. Let  $U(\cat{O})$ denote the underlying category of $\cat{O}$ (i.e. the category of unary operations). Very often, one wishes to lift localizations defined on the level of presheaves, $[U(\cat{O}),\cat{V}] \to [U(\cat{O}),\cat{V}]^{loc}$, to localizations defined on the level of $\cat{O}$-algebras $\algo \to \algo^{loc}$. 
For instance, this is relevant when one wishes to invert operations of $\cat{O}$ as in the setting of \cite{tillmann1, tillmann2}.
This situation can be modeled as a lift of Cisinski's localization, $[U(\cat{O}),\cat{V}]^{loc}$ of locally constant presheaves \cite{cisinski-locally-constant}, to the level of $\cat{O}$-algebras. 
When $\cat{O}$ encodes $n$-operads \cite{batanin-baez-dolan-via-semi}, lifting this localization to the category of $n$-operads would allow for a proof of the generalized Baez-Dolan Stabilization Hypothesis  \cite{batanin-white-baez-dolan}. 
A dual situation \cite{white-yau4} is relevant to proving the McClure-Smith Conjecture \cite{mcclure-smith}, by lifting a right Bousfield localization on $[U(\cat{O}),\cat{V}]$. 

Theorem \ref{thm:omnibus} provides numerous conditions under which such lifted localizations exist. The situation is complicated by the fact that the monad $T=UF$ in the square below is not given by any operad, as we explain below:

\begin{align} \label{diagram:mid-level-loc}
\xymatrix{
\algo \ar@<2.5pt>[r] \ar@<2.5pt>[d]^U
& \algo^{loc} \ar@<2.5pt>[l] \ar@<2.5pt>[d]^U \\
[U(\cat{O}),\cat{V}] \ar@<2.5pt>[r] \ar@<2.5pt>[u]^F  \ar@<2.5pt>[d]^-{R}
& [U(\cat{O}),\cat{V}]^{loc} \ar@<2.5pt>[l] \ar@<2.5pt>[u]^-{F} \\
[U(\cat{O})_d,\cat{V}]  \ar@<2.5pt>[u]^L
&  \\
}
\end{align}

Here $U(\cat{O})_d$ is the discrete category on $U(\cat{O})$, i.e. $U(\cat{O})_d$ is the set of colors $C$ for $\cat{O}$, so $[U(\cat{O})_d,\cat{V}]$ is the product $\cat{V}^{C}$. The functor $R$ is restriction, and $L$ is a left Kan extension. The composition $F\circ L$ is the free $\cat{O}$-algebra functor, and very often one can transfer a (semi-)model structure along this functor to $\algo$. However, the monad $T = UF$ is not given by a $\cat{V}$-operad, so $\algo^{loc}$ cannot be viewed as a category of algebras over an operad valued in $[U(\cat{O}),\cat{V}]^{loc}$. In fact, this monad is the free commutative monoid monad with respect to the Day-Street convolution product $\otimes_{\cat{O}}$ on the category of presheaves $[U(\cat{O}),\cat{V}]$. This product is defined by the following formula \cite{day-street}, \cite{batanin-berger-markl} (Appendix A), where $X_1,\dots,X_n \in [U(\cat{O}),\cat{V}]$ and $c\in U(\cat{O})$:

\[
\otimes_{\cat{O}}(X_1,\dots,X_n)(c) = \int^{c_1,\dots,c_n \in U(\cat{O})} \cat{O}(c_1,\dots,c_n;c) \otimes X(c_1)\otimes \dots \otimes X(c_n)
\]

Because the localization is defined in the middle row of (\ref{diagram:mid-level-loc}), there is no localization on the bottom row that can be lifted through a free operad algebra functor. When $\cat{O}$ satisfies the conditions of Theorem \ref{thm:implies-reflects-local}, $U$ reflects local equivalences, which is enough for many applications. Whether $U$ preserves local equivalences is an interesting question, which we believe does not have an affirmative answer in general.
\end{example}


\section*{Acknowledgments}

We would like to thank the NSF EAPSI program for facilitating our collaboration, Macquarie University for hosting the second author during the summer of 2014 when this work began, and the Program on Higher Structures in Geometry and Physics at the Max Planck Institute for hosting the authors in the spring of 2016 when this work was completed. We are also grateful to Carles Casacuberta, Javier Guti\'{e}rrez, and Markus Spitzweck for numerous helpful conversations, and to Donald Yau for spotting an error in an early draft of this paper.


\end{document}